\documentclass[a4paper,11pt]{article}
\input{xypic}
\newcommand{\quotes}[1]{``#1''}
\usepackage{float,caption}
\usepackage{amsmath, amsthm, amssymb}
\usepackage[top=1in, bottom=1in, left=1in, right=1in]{geometry}
\usepackage{graphicx} 
\usepackage{epstopdf} 
\usepackage{mdframed}
\usepackage{ragged2e}
\usepackage{mathrsfs}
\usepackage[shortlabels]{enumitem}
\usepackage{tikz-cd}
\usepackage{setspace}
\usepackage{stmaryrd}

\newtheorem{thm}{Theorem}[section]

\newtheorem{lem}[thm]{Lemma}
\newtheorem{prop}[thm]{Proposition}
\newtheorem{defn}[thm]{Definition}
\newtheorem{rem}[thm]{Remark}
\newtheorem{exam}[thm]{Example}

\theoremstyle{definition}
\def\Int{\operatorname{Int}}
\def\Cl{\operatorname{Cl}}
\def\P{\mathcal{P}}
\def\T{\mathfrak{T}}

\def\I{\mathfrak{I}}
\def\J{\mathfrak{T}}

\begin{document}
	\setcounter{page}{1}
	\title{Soft Maximal Topologies}
	\author{Samer Al Ghour$^1$ and Zanyar A. Ameen$^2$\\
		$^1$Department of Mathematics and Statistics, Jordan University of Science and Technology,\\ Irbid 22110, Jordan;
		algore@just.edu.jo\\
		$^2$Deptartment of Mathematics, College of Science, University of Duhok,\\ Duhok 42001, Iraq, zanyar@uod.ac}
	\date{}
	\maketitle
	
	\begin{abstract}
In various articles, it is said that the class of all soft topologies on a common universe forms a complete lattice, but in this paper, we prove that it is a complete lattice. Some soft topologies are maximal and some are minimal with respect to specific soft topological properties. We study the properties of soft compact and soft connected topologies that are maximal. Particularly, we prove that a maximal soft compact topology has identical families of soft compact and soft closed sets. We further show that a maximal soft compact topology is soft $T_1$, while a maximal soft connected topology is soft $T_0$. Lastly, we establish that each soft connected relative topology to a maximal soft connected topology is maximal.

$\mathfrak{I}=\mathfrak{T}$
	\end{abstract}

	\section{Introduction}\label{sec1}
    The real world is far too complex for our instant comprehension. We construct \quotes{models} of reality that are simplified versions of reality. Unfortunately, these mathematical models are just too complex, and we are unable to obtain exact solutions. Traditional classical methods are ineffective for modeling problems in engineering, physics, computer sciences, economics, social sciences, medical sciences, and many other domains due to the unpredictability of data. This could be owing to the unpredictability of natural environmental occurrences, human knowledge of the real world, or the limitations of measurement tools. For instance, ambiguity or confusion on the border between states or between major cities, the precise population growth rate in a country's rural areas, or making judgments in a machine-based environment using database information. As a result, classical set theory, which is predicated on the crisp and accurate case, may not be totally adequate for dealing with such uncertainty concerns. Theory of fuzzy sets \cite{zadeh1996fuzzy}, theory of intuitionistic fuzzy sets \cite{atanassov1986intuitionistic}, the theory of vague sets \cite{gau1993vague}, the theory of interval mathematics \cite{atanassov1994operators}, and the theory of rough sets \cite{pawlak1995rough} are some of the theories. These theories might be seen as instruments for dealing with uncertainty, but each has its own set of problems. The insufficiency of the theory's parametrization tool, as highlighted by Molodtsov in \cite{molodtsov1999soft} could be the cause of these difficulties. He invented the word \quotes{soft set theory} to describe a new mathematical tool that is free of the issues discussed above. He stated the core results of the new theory in his paper \cite{molodtsov1999soft}, and effectively applied it to a variety of fields, including smoothness of functions, game theory, operations research, Riemann-integration, and probability theory. A \quotes{soft set} is a collection of approximations to an object's description.

General topology is the branch of topology that deals with the fundamental set-theoretic notions and constructions used in topology. It is the foundation of most other topics in topology, including differential topology, geometric topology, and algebraic topology. Soft topology, which combines soft set theory and topology, is another field of topology. It is concerned with a structure on the set of all soft sets and is motivated by the standard axioms of classical topological space. The work of Shabir and Nazs \cite{shabir2011soft}, in particular, was crucial in establishing the field of soft topology. After that various classes of soft topological spaces have been proposed, such as: soft compact \cite{aygunouglu2012some}, soft connected \cite{lin2013soft}, soft paracompact \cite{lin2013soft}, soft extremally disconnected \cite{asaad2017results}, and soft separable spaces \cite{bayramov2018new}, soft J-spaces \cite{mohammed2019soft}, soft Menger spaces \cite{al2020nearly} and soft separation axioms \cite{al2021weaker,bayramov2018new}. At this point, it is worth remarking that not all classical results in topology are true in soft topology, see Theorem 4 in \cite{al2019comment}.  Introducing all the above terminologies, arguments, and Remark \ref{z} motivate us to study the structure of maximal soft compact and maximal soft connected topologies. 

This paper is organized as follows: Sections 1\&2 are dedicated to a brief introduction and preliminary concepts from soft set theory and soft topology. Section 3 starts by showing that the set of all soft topologies on a common universe forms a complete lattice. The definition of a maximal soft topology with property $\Gamma$ is given, followed by two subsections. The first one defines the concept of a maximal soft compact topology. Some properties and a characterization of maximal soft compact topologies are established. The second subsection concerns the fundamental properties of maximal soft connected topologies. It also contains some examples that present the structure of maximal soft connected topologies. Finally, the main result on maximal soft connected topologies is demonstrated. Section 4 concludes the summary of our findings and proposes possible lines for future work.

	\section{Preliminaries}\label{se2}\
	Let $Z$ be an initial universe, $\mathcal{P}(Z)$ be all subsets of $Z$ and $E$ be a set of parameters. An ordered pair $(Y,E)=\{(e,Y(e)):e\in E\}$ is said to be a soft set over $Z$, where $Y:E\to\P(Z)$ is a set-valued mapping. The family of all soft sets on $Z$ is represented by $\P(Z,E)$. A soft element \cite{nazmul2013neighbourhood} is a soft set $(Y,E)$ over $Z$ in which $(Y,E) = \{z\}$ for all $e\in E$, where $z\in Z$, and is denoted by $(\{z\},E)$.
	A soft point \cite{allam2017new}, denoted by $z(e)$, is a soft set $(Y,E)$ over $Z$ in which $Y(e) = \{z\}$ and $Y(e')= \emptyset$ for each $e'\neq e$, $e'\in E$, where $e\in E$ and $z\in Z$. A statement $z(e)\widetilde{\in} (Y,E)$ means that $z\in Y(e)$. The singleton soft set $\{z(e)\}$ is referred to $\{z(e)\}=\{(e,\{z\}),(e',\emptyset),\cdots:\forall e'\in E, e'\ne e\}$. The soft set $(Z,E)\backslash (Y,E)$ $($or simply $(Y,E)^c)$ is the complement of $(Y,E)$, where $Y^c:E\to\mathcal{P}(Z)$ is given by $Y^c(e) = Z\backslash Y(e)$ for all $e\in E$. A soft subset $(Y,E)$ over $Z$ is called null, denoted by $\widetilde{\Phi}$, if $Y(e)=\emptyset$ for any $e\in E$ and called absolute, denoted by $\widetilde{Z}$, if $Y(e) = Z$ for any $e\in E$. Notice that ${\widetilde{Z}}^c=\widetilde{\Phi}$ and $\widetilde{\Phi}^c=\widetilde{Z}$. It is said that $(X,E_1)$ is a soft subset of $(Y,E_2)$ (written by $(X,E_1)\widetilde{\subseteq} (Y,E_2)$, \cite{maji2003soft}) if $E_1\subseteq E_2$ and $X(e)\subseteq Y(e)$ for any $e\in E_1$. We say $(X,E_1)=(Y,E_2)$ if $(X,E_1)\widetilde{\subseteq} (Y,E_2)$ and $(Y,E_2)\widetilde{\subseteq} (X,E_1)$.

\begin{defn}\cite{terepeta2019separating,ali2009some}
		Let $\{(Y_i,E):i\in I\}$ be a family of soft sets over $Z$, where $I$ is any index set.
		\begin{enumerate}[(i)]
			\item The intersection of $(Y_i,E)$, for $i\in I$, is a soft set $(Y,E)$ such that $Y(e)=\bigcap_{i\in I}Y_i(e)$ for each $e\in E$ and denoted by $(Y,E)=\widetilde{\bigcap}_{i\in I}(Y_i,E)$.
			\item The union of $(Y_i,E)$, for $i\in I$, is a soft set $(Y,E)$ such that $Y(e)=\bigcup_{i\in I}Y_i(e)$ for each $e\in E$ and denoted by $(Y,E)=\widetilde{\bigcup}_{i\in I}(Y_i,E)$.
		\end{enumerate}
\end{defn}

	\begin{defn}\cite{shabir2011soft}
		A collection $\I$ of $\P(Z,E)$ is said to be a soft topology on $Z$ if the following conditions are satisfied
		\begin{enumerate}[(i)]
			\item $\widetilde{\Phi},\widetilde{Z}\widetilde{\in}\I$;
			\item If $(Y_1,E), (Y_2,E)\widetilde{\in}\I$, then $(Y_1,E)\widetilde{\cap} (Y_2,E)\widetilde{\in}\I$; and
			\item If any $\{(Y_i,E):i\in I\}\widetilde{\subseteq}\I$, then $\widetilde{\bigcup}_{i\in I} (Y_i,E)\widetilde{\in}\I$.
		\end{enumerate}
		Terminologically, we call $(Z, \I, E)$ a soft topological space on $Z$. The elements of $\I$ are called soft $\I$-open sets (or simply soft open sets when no confusion arise), and their complements are called soft $\I$-closed sets (or soft closed sets).
	\end{defn}
	
	In what follows, by $(Z, \I, E)$ we mean a soft topological space, by two distinct soft points $u(e), v(e')$ we mean either $u\ne v$ or $e\ne e'$, and by two disjoint soft sets $(Y,E), (X,E)$ over $Z$, we mean $(Y,E)\widetilde{\cap}(X,E)=\widetilde{\Phi}$.
	\begin{defn}\cite{ccaugman2011soft}
		A subcollection $\mathcal B\subseteq\I$ is called a soft base for the soft topology $\I$ if each element of $\I$ is a union of elements of $\mathcal B$.
	\end{defn}
	
	\begin{defn}\cite{shabir2011soft}
		Let $(Y,E)\ne\widetilde{\Phi}$ be a soft subset of $(Z, \I, E)$. Then $\I_Y:=\{(G,E)\widetilde{\bigcap} (Y,E):(G,E)\widetilde{\in}\I\}$ is called a soft relative topology over $Y$ and $(Y, \I_Y, E)$ is a soft subspace of $(Z, \I, E)$.
	\end{defn}
	
	\begin{defn}\cite{shabir2011soft}
		Let $(Y,E)$ be a soft subset of $(Z, \I, E)$. The soft interior of $(Y,E)$, denoted by $\Int_{\I}((Y,E))$, is the largest soft open set contained in $(Y,E)$. The soft closure of $(Y,E)$, denoted by $\Cl_{\I}((Y,E)))$, is the smallest soft closed set which contains $(Y,E)$. The soft closure and interior of a soft subset $(Y,E)$ in the subspace $(X, \I_X, E)$ are respectively denoted by $\Cl_X((Y,E)))$ and $\Int_X((Y,E))$
	\end{defn}
	
	\begin{lem}\cite{hussain2011some}
		For a soft subset $(Y,E)$ of $(Z, \I, E)$, $$\Int_\I((Y,E)^c)=(\Cl_\I((Y,E)))^c \text{ and } \Cl_\I((Y,E)^c)=(\Int_\I((Y,E)))^c.$$
	\end{lem}

\begin{defn}\label{softmap}\cite{kharal2011mappings,zorlutuna2012remarks}
	Let $(Z,E), (Y,E')$ be soft sets, and let $p:Z\to Y, q:E\to E'$ be functions. The image of a soft set $(A,E)\widetilde{\subseteq}(Z,E)$ under $f:(Z,E)\to (Y,E')$ is a soft subset $f(A,E)=(f(A),q(E))$ of $(Y,E')$  which is given by 
	$$f(A)(e')=
	\begin{cases}
		\bigcup\limits_{e\in q^{-1}(e')\cap E} p\left(A(e)\right),& q^{-1}(e')\cap E\ne\emptyset\\
		\emptyset,& \text{ otherwise},
	\end{cases}
	$$
	for each $e'\in E'$.
	
	The inverse image of a soft set $(B,E')\widetilde{\subseteq}(Y,E')$ under $f$ is a soft subset $f^{-1}(B,E')=(f^{-1}(B),q^{-1}(E'))$ such that 
	
	$$(f^{-1}(B)(e)=
	\begin{cases}
		p^{-1}\left(B(q(e))\right),& q(e)\in E'\\
		\emptyset,& \text{ otherwise},
	\end{cases}
	$$
	for each $e\in E$.
	
	The soft function $f$ is bijective if both $p$ and $q$ are bijective.
\end{defn}

\begin{lem}\label{zlem}
	Let $(Z,E), (Y,E')$ be soft sets. If $f:(Z,E)\to (Y,E')$ is bijection, then $f((A,E)^c)=(f(A,E))^c$ for each $(A,E)\widetilde{\subseteq}(Z,E)$.
\end{lem}
\begin{proof}
It follows from Theorem 3.14 in \cite{zorlutuna2012remarks}.
\end{proof}

\begin{defn}\cite{nazmul2013neighbourhood}
	Let $(Z, \I, E)$ and $(Y, \J, E')$ be soft topological spaces. A soft function $f:(Z, \I, E)\to(Y, \J, E')$ is said to be
	\begin{enumerate}[(i)]
		\item soft continuous  if the inverse image of each soft open subset of $(Y, \J, E')$ is a soft open subset of $(Z, \I, E)$.
		\item soft open if the image of each soft open subset of $(Z, \I, E)$ is a soft open subset of $(Y, \J, E')$.
		\item soft homeomorphism if it is a soft open and soft continuous bijection from $(Z, \I, E)$ onto $(Y, \J, E')$.
	\end{enumerate}
\end{defn}

\section{Maximal soft topologies}
\begin{defn}
Let $\I_1, \I_2$ be two soft topologies on $Z$. It is said that $\I_1$ is coarser than $\I_2$ if $\I_1\widetilde{\subset}\I_2$. And $\I_1$ is finer than $\I_2$ if  $\I_2\widetilde{\subset}\I_1$.
\end{defn}
\begin{lem}\label{intersectiontopologies}
Let $\mathcal{F}=\{\I_\alpha:\alpha\in \Lambda\}$ be a family of soft topologies on $Z$, where $\Lambda$ is any index set. Then $\I=\widetilde{\bigcap}_{\alpha\in\Lambda}\I_\alpha$ is a soft topology on $Z$.
\end{lem}
\begin{proof}
Evidently, $\widetilde{\Phi}, \widetilde{Z}$ belong to $\I$ as they belong to $\I_\alpha$ for all $\alpha$s. Let $(Y_1,E), (Y_2,E)\widetilde{\in}\I$. Then $(Y_1,E), (Y_2,E)\widetilde{\in}\I_\alpha$ for all $\alpha\in\Lambda$. Since all $\I_\alpha$ are soft topologies, so $(Y_1,E)\widetilde{\bigcap}(Y_2,E)\widetilde{\in}\I_\alpha$ for all $\alpha\in\Lambda$. Therefore $(Y_1,E)\widetilde{\bigcap}(Y_2,E)\widetilde{\in}\widetilde{\bigcap}_{\alpha\in\Lambda}\I_\alpha=\I$. Let $\{Y_\lambda:\lambda\in \Delta\}$ be a collection of sets in $\I$. Then for each $\alpha\in\Lambda$, $Y_\lambda\widetilde{\in}\I_\alpha$ for all $\lambda\in \Delta$. But for each $\alpha$, $\I_\alpha$ is a soft topology on $Z$, so $\widetilde{\bigcup}_{\lambda\in \Delta}Y_\lambda\widetilde{\in}\I_\alpha$ for all $\alpha\in\Lambda$. Hence $\widetilde{\bigcup}_{\lambda\in \Delta}Y_\lambda\widetilde{\in}\widetilde{\bigcap}_{\alpha\in\Lambda}\I_\alpha=\I$.
\end{proof}

The above result is an extension of Proposition 6 in \cite{shabir2011soft}.

\begin{lem}\label{intincludingF}
Let $\mathcal{C}$ be a collection of soft subsets over $Z$. There exists a unique soft topology $\I$ on $Z$ including $\mathcal{C}$ and if $\I_0$ is any other soft topology on $Z$ that includes $\mathcal{C}$, then $\I\widetilde{\subset}\I_0$.
\end{lem}
\begin{proof}
Notice that such a soft topology always exists because $\P(Z,E)$ is the soft topology on $Z$ which includes $\mathcal{C}$. Consider $\I$, the intersection of all those soft topologies on $Z$ which include $\mathcal{C}$. Then it follows from Lemma \ref{intersectiontopologies} that $\I$ is the required soft topology.
\end{proof}

\begin{defn}\label{generatingtopology}
Let $\mathcal{C}$ be a collection of soft subsets over $Z$. The unique soft topology obtained in the above lemma is called the soft topology on $Z$ generated by the collection $\mathcal{C}$ and is denoted by $\I(\mathcal{C})$, which is the smallest soft topology on $Z$ including $\mathcal{C}$.
\end{defn}

The union of two soft topologies is not a soft topology, see Example 3 in \cite{shabir2011soft}, but we can generate a unique soft topology that includes both of them.
\begin{lem}\label{join}
Let $\I_1, \I_2$ be two soft topologies on $Z$. The generating soft topology $\I(\I_1\widetilde{\cup} \I_2)$ is identical to the soft topology $\I(\mathcal{F})$ generated by $\mathcal{F}=\{(G_1,E)\widetilde{\cap}(G_2,E):(G_1,E)\widetilde{\in}\I_1,(G_2,E) \widetilde{\in}\I_2\}$.
\end{lem}
\begin{proof}
Since $\I_1, \I_2$ soft topologies, so they include $\widetilde{Z}$. By taking $\widetilde{Z}=(G_i,E)$, for $i=1,2$, then $\mathcal{F}$ will exactly contain $\I_1, \I_2$. By the uniqueness of the generating soft topology, $\I(\mathcal{F})=\I(\I_1\widetilde{\cup} \I_2)$, see Definition \ref{generatingtopology}.
\end{proof}

\begin{thm}
The set $\Sigma$ of all soft topologies over a common universe $Z$ forms a complete lattice under soft set inclusion \quotes{$\widetilde{\subset}$}.
\end{thm}
\begin{proof}
One can easily show that $\widetilde{\subset}$ is a partially ordered set on $\Sigma$. It remains to prove that every subset of $\Sigma$ has the greatest lower bound and the least upper bound. Let $\Sigma_0$ be a subset of $\Sigma$. By Lemma \ref{intersectiontopologies}, $\bigwedge\Sigma_0=\widetilde{\bigcap}\{\I_0:\I_0\widetilde{\in}\Sigma_0\}$ is the greatest lower bound of $\Sigma_0$. By Lemma \ref{join}, $\bigvee\Sigma_0=\I(\widetilde{\bigcup}\{\I_0:\I_0\widetilde{\in}\Sigma_0\})$ is the least upper bound of $\Sigma_0$.
\end{proof}

\begin{rem}\label{z}
Note that the indiscrete soft topology $\I_I$ on $Z$ is the minimal (smallest) element in $\Sigma$ and the discrete soft topology $\I_D$ on $Z$ is the maximal (largest) element in $\Sigma$. It is worth remarking that $\I_D$ is the maximal soft Hausdorff topology and $\I_I$ is the minimal soft compact (and minimal soft connected) topology. From the latter statement, we understand that maximal covering and connectedness properties are more interesting to study. On the other hand, minimal separation axioms are more interesting. Hence, we focus on considering maximal soft compact and maximal soft connected spaces.
\end{rem}

\begin{defn}\label{defnMax}
Let $(Z,\I,E)$ be a soft topological space with the property $\Gamma$. Then $\I$ is called $\Gamma$-maximal if any soft topology finer than $\I$ does not have the property $\Gamma$.
\end{defn}

\subsection{Maximal soft compact topologies}
Recall that a space $(Z, \I, E)$ is called soft compact \cite{aygunouglu2012some} if each soft open cover of $\widetilde{Z}$ has a finite subcover. If we replace $\Gamma$ by soft compactness, the Definition \ref{defnMax} will be
\begin{defn}\label{Maxcompact}
Let $(Z,\I,E)$ be a soft compact space. Then $\I$ is called maximal soft compact if any soft topology finer than $\I$ is not soft compact.
\end{defn}

The following example shows the structure of maximal soft compact spaces.

\begin{exam}\label{ex2}
Consider the set of naturals $\mathbb{N}$ and $E=\{e\}$. Define a soft topology $\I=\{Y_E\widetilde{\subseteq}\widetilde{\mathbb{N}}:Y_E=\{n(e)\} \text{ for }n\ne 1\text{ or } 1(e)\widetilde{\in}Y_E \text{ if } Y^c_E\text{ is finite}\}\widetilde{\bigcup}\{\widetilde{\Phi}\}$. Then $(\mathbb{N}, \I, E)$ is maximal soft compact.
\end{exam}

\begin{lem}\label{closedincompactiscompact}\cite[Proposition 5.1]{el2018partial}
 Let $(Z, \I, E)$ be a soft compact space and let $(Y, E)\widetilde{\subseteq}\widetilde{Z}$. If $(Y, E)$ is soft closed, then $(Y, E)$ is soft compact.
 \end{lem}

\begin{defn}\cite{ameen2021alghour}
Let $(Z,\I,E)$ be any soft topological space and let $(Y, E)$ be any soft non-open subset over $Z$. The soft topology $\I^*$ on $Z$ generated by $\I\widetilde{\cup}\{(Y, E)\}$ is said to be an s-extension of $\I$ and it is denoted by $\I^*=\I[(Y,E)]$ (or shortly, $\I^*=\I[Y]$).
\end{defn}

\begin{lem}\label{compactthm}\cite[Theorem 3.2]{ameen2021alghour}
Let $(Z, \I, E)$ be a soft compact space. Then $(Z, \I[Y], E)$ is soft compact if and only if $(Y^c,E)$ is soft compact in $(Z, \I, E)$.
\end{lem}

\begin{prop}\label{maxcomp=CK}
A soft topological space $(Z,\I,E)$ is maximal soft compact if and only if the family of all soft $\I$-closed sets is equal to the family of all soft $\I$-compact sets.
\end{prop}
\begin{proof}
Assume that $\I$ is maximal soft compact. If $(Y,E)$ is a soft compact set but not soft closed, then  $(Y^c,E)$ is not soft open. By Lemma \ref{compactthm}, $\I[Y^c]$ is soft compact. But $\I\widetilde{\subset}\I[Y^c]$, so this violates the maximality of $\I$. Hence $(Y,E)$ must be closed.

Conversely, suppose the family of all soft $\I$-closed sets is equal to the family of all soft $\I$-compact sets. If $\I$ is not maximal soft compact, there exists a soft compact topology $\I^*$ such that $\I\widetilde{\subset}\I^*$. Therefore there is a set $(Y,E)\widetilde{\subset}\widetilde{Z}$ which is soft $\I^*$-closed but not soft $\I$-closed. By assumption, $(Y,E)$ is not $\I$-compact. Therefore, there exists a soft $\I$-open cover $\mathscr{U}$ of $(Y,E)$ which has no finite subcover. Since $\I\widetilde{\subset}\I^*$, so $\mathscr{U}$ is also a soft $\I^*$-open cover of $(Y,E)$. This means that $(Y,E)$ is not soft $\I^*$-compact. But this is a contradiction, because $(Y,E)$ is soft $\I^*$-closed and, by Lemma \ref{closedincompactiscompact}, it must be soft $\I^*$-compact. This proves that $\I$ has to be maximal soft compact. 
\end{proof}

\begin{prop}\label{maxcomp=con=homeo}
A soft topological space $(Z,\I,E)$ is maximal soft compact if and only if each soft continuous bijection from a soft compact space onto $(Z,\I,E)$ is a soft homeomorphism.
\end{prop}
\begin{proof}
Assume $(Z,\I,E)$ is maximal soft compact. Let $f:(Y,\I',E')\to(Z,\I,E)$ be a soft continuous bijection, where $(Y,\I',E')$ is soft compact. Take $\T=\{f((U,E')):(U,E')\widetilde{\in}\I'\}$. Evidently $(Z,\T,E)$ is a soft topology and $f$ is a soft homeomorphism. Since $(Y,\I',E')$ is soft compact, $(Z,\T,E)$ is also soft compact under $f$. But $\I\widetilde{\subseteq}\T$ and $\I$ is maximal, thus $\T=\I$.

Conversely, if $(Z,\I,E)$ is not maximal soft compact, then there exists a soft compact topology $\I^*$ on $Z$ such that $\I\widetilde{\subseteq}\I^*$. Then the identity soft function $h:(Z,\I^*,E)\to(Z,\I,E)$ is a soft continuous bijection but not a soft homeomorphism. This completes the proof.
\end{proof}

\begin{thm}\label{maxcompactthm}
For a soft topological space $(Z,\I,E)$, the following are equivalent:
\begin{enumerate}[(1)]
	\item $\I$ is maximal soft compact;
	\item the family of all soft $\I$-closed sets is equal to the family of all soft $\I$-compact sets; and
	\item each soft continuous bijection from a soft compact space onto $(Z,\I,E)$ is a soft homeomorphism.
\end{enumerate} 
\end{thm}
\begin{proof}
(1)$\implies$(2) Proposition \ref{maxcomp=CK}

(2)$\implies$(3) Suppose $f$ is a soft continuous bijection from a soft compact space $(Y,\I',E')$ onto $(Z,\I,E)$. It remains to check that $f$ is soft open. Let $(G,E')$ be soft $\I'$-open. Then $(G^c,E')$ is soft $\I'$-closed. Since $(Y,\I',E')$ is soft compact space, by Lemma, \ref{closedincompactiscompact}, $(G^c,E')$ is soft $\I'$-compact. From the soft continuity of $f$, $f((G^c,E'))$ is soft $\I$-compact and by (2) $f((G^c,E'))$ is soft $\I$-closed. Since $f$ is bijective, $f((G^c,E'))=(f((G,E')))^c$ and so $f((G,E'))$ is soft open. Hence $f$ is a soft open function.

(3)$\implies$(1) Proposition \ref{maxcomp=con=homeo}.
\end{proof}

\begin{lem}\label{closed=t1}\cite[Theorem 4.1]{bayramov2018new}
Let $(Z,\I,E)$ be a soft topological space. If each singleton soft set is soft closed, then $(Z,\I,E)$ is soft $T_1$.
\end{lem}

The above is true in various soft point theories, see \cite{al2020comments}

\begin{thm}\label{maxcomcor}
If $(Z,\I,E)$ is a maximal soft compact space, then $(Z,\I,E)$ is soft $T_1$.
\end{thm}
\begin{proof}
One can easily show that each singleton soft set is soft compact. Since $(Z,\I,E)$ is maximal soft compact, by Proposition \ref{maxcomp=CK}, each singleton soft set is soft closed and Lemma \ref{closed=t1} finishes the proof.
\end{proof}

\begin{defn}\cite{al2019comment}
A soft set $(A,E)$ from $(Z,\I,E)$ is called stable if there exists a subset $Y$ of $Z$ such that $A(e)=Y$ for each $e\in E$.
\end{defn}

\begin{defn}
We call a soft topological space $(Z,\I,E)$ stable if each soft open is stable. 
\end{defn}

\begin{lem}\label{stable}
If $(Z,\I,E)$ is a stable soft $T_2$-space, then each soft compact is soft closed.
\end{lem}
\begin{proof}
It follows from Lemma 7 and Theorem 8 in \cite{al2019comment}.
\end{proof}

\begin{thm}
If $(Z,\I,E)$ is a stable soft compact $T_2$-space, then $(Z,\I,E)$ is maximal soft compact.
\end{thm}
\begin{proof}
If  $\I$ is not maximal soft compact, there exists a soft compact topology $\I^*$ on $Z$ such that $\I\widetilde{\subset}\I^*$. Pick a set $(Y,E)$ to be soft $\I^*$-open but not soft $\I$-open. Let $\I_0=\I[Y]$. Then $\I\widetilde{\subset}\I_0\widetilde{\subset}\I^*$ and so $\I_0$ is soft compact. By Lemma \ref{compactthm}, $(Y^c,E)$ is a soft $\I$-compact set and by Lemma \ref{stable}, it is soft $\I$-closed. This implies  $(Y,E)$ is soft $\I$-open, which contradicts to the choice of $(Y,E)$. Hence $(Z,\I,E)$ is maximal soft compact.
\end{proof}

\subsection{Maximal soft connected topologies}
\begin{defn}\cite{lin2013soft}
A soft topological space $(Z, \I, E)$ is called soft connected if it cannot be written as a union of two disjoint soft open sets. Otherwise, it called soft disconnected.
\end{defn}
\begin{defn}\label{Maxconnected}
Let $(Z,\I,E)$ be a soft connected space. Then $\I$ is called maximal soft connected if any soft topology finer than $\I$ is not soft connected. 
\end{defn}

We start by giving some examples of maximal soft connected spaces.
\begin{exam}\label{ex3}
Let $\mathbb{R}$ be the set of reals and let $E=\{e_1,e_2\}$. The soft topological space $(\mathbb{R}, \I, E)$ is maximal soft connected, where $\I=\{Y_E\widetilde{\subseteq}\widetilde{\mathbb{R}}:0(e_1)\widetilde{\in}Y_E\}\widetilde{\bigcup}\{\widetilde{\widetilde{\Phi}}\}$.
\end{exam}

\begin{exam}\label{ex1}\cite[Example 3.3]{ameen2021alghour}
The soft topological space $(\mathbb{R}, \I, E)$ is maximal soft connected, where $\I=\{Y_E\widetilde{\subseteq}\widetilde{\mathbb{R}}:0(e_1)\widetilde{\notin}Y_E\}\widetilde{\bigcup}\{\widetilde{\mathbb{R}}\}$, $\mathbb{R}$ is the set of reals, and $E=\{e_1,e_2\}$. 
\end{exam}

\begin{rem}
From the above example, we shall remark that the maximal soft topology dealt with in this note does not have a nice connection with maximal crisp topology \cite{vaidya1947treatise}, in general, due to the concept of soft point we select. The soft topology given in Example \ref{ex1} is maximal and the crisp topology of $e_1$ is maximal, while the crisp topology of $e_2$ is not maximal.
\end{rem} 

\begin{thm}\label{maxconnthm}
Let $(Y,E)$ be a soft subset of a soft topological space $(Z,\I,E)$. If $(Y,E), (Y^c,E)$ are soft connected (as soft subspaces) and either of them is soft open, then $\I$ is maximal soft connected.
\end{thm}
\begin{proof}
Assume $(Y,E)\ne\widetilde{\Phi}\ne(Y^c,E)$, otherwise the result trivially holds. Suppose that $(Y,E), (Y^c,E)$ are not soft $\I$-open. Then, by taking $\I^*=\I[Y]$, we obtain a disconnected soft topology $\I^*$ such that $\I\widetilde{\subset}\I^*$. Therefore, there are disjoint soft $\I^*$-open sets $(G,E), (H,E)$ such that $\widetilde{Z}=(G,E)\widetilde{\bigcap}(H,E)$. W.L.O.G assume that $(G,E)\widetilde{\cap}(Y^c,E)\ne\widetilde{\Phi}\ne(H,E)\widetilde{\cap}(Y^c,E)$. Then $(G,E)\widetilde{\cap}(Y^c,E), (H,E)\widetilde{\cap}(Y^c,E)$ are disjoint soft $\I^*_{Y^c}$-open and $(G,E)\widetilde{\cap}(Y^c,E)\widetilde{\bigcup} (H,E)\widetilde{\cap}(Y^c,E)\\=(Y^c,E)$. By Remark 2.2 (vii) in \cite{ameen2021alghour}, $(G,E)\widetilde{\cap}(Y^c,E), (H,E)\widetilde{\cap}(Y^c,E)$ are disjoint soft $\I_{Y^c}$-open which means that $(Y^c,\I_{Y^c}, E)$ is not soft connected, a contradiction. Therefore, either $(G,E)\widetilde{\subseteq}(Y,E)$ or $(H,E)\widetilde{\subseteq}(Y,E)$. Suppose  $(G,E)\widetilde{\subseteq}(Y,E)$. If $(G,E)=(Y,E)$, then $(H,E)=(Y^c,E)$, which is a soft $\I$-open set, but that is not possible (as it is soft $\I^*$-open). Therefore, we must have $(H,E)\widetilde{\cap}(Y,E)\ne\widetilde{\Phi}$. Hence $(G,E)\widetilde{\bigcup} (H,E)\widetilde{\cap}(Y,E)=(Y,E)$, where $(G,E), (H,E)\widetilde{\cap}(Y,E)$ are soft $\I^*_Y$-open. But, by Remark 2.2 (vii) in \cite{ameen2021alghour} soft $\I^*_Y$-open and soft $\I_Y$-open are similar, thus $(Y,\I_Y,E)$ is soft disconnected, again contradiction. The result is proved.
\end{proof}

\begin{thm}
Let $(Z,\I,E)$ be a soft connected space. If $\I$ is maximal, then $(Z,\I,E)$ is soft $T_0$. 
\end{thm}
\begin{proof}
Suppose $(Z,\I,E)$ is not soft $T_0$. Then there are soft points $y(e),z(e')$ with $y(e)\ne z(e')$ such that $y(e),z(e')\widetilde{\in}(G,E)$ for all soft $\I$-open sets $(G,E)$. Therefore, $y(e)\widetilde{\in}\Cl(\{z(e')\})$ and $z(e')\widetilde{\in}\Cl(\{y(e)\})$. Let $\I^*=\I[\{z(e')\}]$. Then $\I\widetilde{\subset}\I^*$ and so $\I^*$ is not soft connected as $\I$ is maximal soft connected. Therefore, there exist disjoint soft $\I^*$-open sets $(G,E),(H,E)$ such that $(G,E)\widetilde{\cup}(H,E)=\widetilde{Z}$. Thus either $y(e),z(e')\widetilde{\in}(G,E)$ or $y(e),z(e')\widetilde{\in}(H,E)$. If $y(e),z(e')\widetilde{\in}(G,E)$, by Remark 2.2 (ii) in \cite{ameen2021alghour}, each soft $\I^*$-open set containing $y(e)$ is also soft $\I$-open, so $(G,E)$ is soft $\I$-open. Therefore, there is a soft $\I$-open $(U,E)$ such that $y(e)\widetilde{\in}(U,E)\widetilde{\subseteq}(G,E)$. But $(U,E)$ is a soft $\I$-open containing $z(e')$ such that $(U,E)\widetilde{\subseteq}(G,E)$. Thus $(G,E)$ is a soft $\I$-open containing each of its points, and similarly for $(H,E)$. This implies that $(Z,\I,E)$ is soft disconnected, which is impossible.
\end{proof}

\begin{prop}\label{l1}
If $(Y,E)$ is a soft open connected subset of a maximal soft connected space $(Z,\I,E)$, then $(Y,\I_Y,E)$ is maximal soft connected.
\end{prop}
\begin{proof}
If $(Y,\I_Y,E)$ is not maximal soft connected, then there exists a soft connected topology $\T$ on $Y$ such that $\I_Y\widetilde{\subset}\T$. Let $(X,E)\widetilde{\subseteq}(Y,E)$ be a soft $\T$-open set but not soft $\T_Y$-open. If $\T^*=\I_Y[X]$, then $\I_Y\widetilde{\subset}\T^*$ and so $\T^*$ is soft connected. If $\I^*=\I[X]$, then $\I\widetilde{\subset}\I^*$ but $\I^*$ cannot be soft connected. 
Therefore, there are disjoint soft $\I^*$-open sets $(G,E),(H,E)$ such that $(G,E)\widetilde{\cup}(H,E)=\widetilde{Z}$. Then either $(Y,E)\widetilde{\subseteq}(G,E)$ or $(Y,E)\widetilde{\subseteq}(H,E)$, differently $(Y,\T^*,E)=(Y,\I^*_Y,E)$ will be soft disconnected by the disjoint soft $\T^*$-open sets $(G,E)\widetilde{\cap}(Y,E),(H,E)\widetilde{\cap}(Y,E)$ (impossible). Hence, we assume $(Y,E)\widetilde{\subseteq}(G,E)$. By Remark 2.2 (ii), for all soft points $z(e)$ that belongs to some soft $\I^*$-open set but does not belong to any soft $\I$-open, we have $$z(e)\widetilde{\in}(X,E)\widetilde{\subseteq}(Y,E)\widetilde{\subseteq}(G,E).$$ This means that $(G,E)$ is soft $\I$-open as $(Y,E)$ is soft $\I$-open. Similarly, one can show that $(H,E)$ is soft $\I$-open. Hence $(G,E), (H,E)$ are disjoint soft $\I$-open and $(G,E)\widetilde{\cup}(H,E)=\widetilde{Z}$, which proves that $(Z,\I,E)$ is not soft connected, a contradiction.
\end{proof}

\begin{prop}\label{l2}
	If $(Y,E)$ is a soft closed connected subset of a maximal soft connected space $(Z,\I,E)$, then $(Y,\I_Y,E)$ is maximal soft connected.
\end{prop}
\begin{proof}
Assume $(Y,\I_Y,E)$ is not maximal soft connected. Take a set $(X,E)\widetilde{\subset}(Y,E)$ which is not soft $\I_Y$-open. If $\T=\I_Y[X]$, then $\I_Y\widetilde{\subset}\T$ and so $\T$ is soft connected. Let $(W,E)=\widetilde{Z}\backslash((Y,E)\backslash(X,E))$. Then, we have $(W,E)\widetilde{\cap}(Y,E)=(X,E)$. Therefore, $\I_Y[X]=\I^*_Y$, where $\I^*$ is an s-extension of $\I$ with respect to $(W,E)$ (i.e. $\I^*=\I[W]$). Since $\I$ is maximal soft connected and $\I\widetilde{\subset}\I^*$, then $\I^*$ is soft disconnected. Thus there exist disjoint soft $\I^*$-open sets $(G,E), (H,E)$ such that $(G,E)\widetilde{\cup}(H,E)=\widetilde{Z}$. Since $\I^*_Y$ is soft connected, then either $(Y,E)\widetilde{\subset}(G,E)$ or $(Y,E)\widetilde{\subset}(H,E)$. Suppose $(Y,E)\widetilde{\subset}(G,E)$. We consider two cases: (i) suppose $z(e)\widetilde{\in}(G,E)$. Since $(G,E)$ is soft $\I^*$-open, there are soft $\I$-open sets $(U,E), (V,E)$ such that $z(e)\widetilde{\in}(U,E)\widetilde{\cup}[(V,E)\widetilde{\cap}(W,E)]\widetilde{\subset}(G,E)$. If $z(e)\widetilde{\in}(U,E)]\widetilde{\subset}(G,E)$ and so $z(e)\widetilde{\in}\Int_{\I}((G,E))$. If $z(e)\widetilde{\in}(V,E)\widetilde{\cap}(W,E)\widetilde{\subset}(G,E)$, then $z(e)\widetilde{\in}(V,E)=(V,E)\widetilde{\cap}[(\widetilde{Z}\backslash (W,E))\widetilde{\cup}(W,E)]\widetilde{\subset}(G,E)$ as $[(V,E)\widetilde{\cap}(\widetilde{Z}\backslash(W,E))]\widetilde{\subset}\widetilde{Z}\backslash(W,E)=(Y,E)\backslash(X,E)\widetilde{\subseteq}(Y,E)\widetilde{\subseteq}(G,E)$. Again $z(e)\widetilde{\in}\Int_{\I}((G,E))$. Since $z(e)$ was arbitrarily taken, so $(G,E)$ is soft $\I$-open. (ii) suppose $z(e)\widetilde{\in}(H,E)$. Since $(H,E)$ is soft $\I^*$-open, there are soft $\I$-open sets $(U,E), (V,E)$ such that $z(e)\widetilde{\in}(U,E)\widetilde{\cup}[(V,E)\widetilde{\cap}(W,E)]\\ \widetilde{\subset}(H,E)$. If $z(e)\widetilde{\in}(U,E)]\widetilde{\subset}(H,E)$ and so $z(e)\widetilde{\in}\Int_{\I}((H,E))$. If $z(e)\widetilde{\in}(V,E)\widetilde{\cap}(W,E)$, then, since $z(e)\widetilde{\in}(H,E)\widetilde{\subseteq}\widetilde{Z}\backslash (Y,E)\widetilde{\subset}(W,E)$, so $z(e)\widetilde{\in}(V,E)\widetilde{\cap}[\widetilde{Z}\backslash(Y,E)]\widetilde{\subset}(V,E)\widetilde{\cap}(W,E)$. Since $\widetilde{Z}\backslash(Y,E)$ is soft $\I$-open, therefore $z(e)\widetilde{\in}\Int_{\I}((H,E))$. Thus $(H,E)$ is soft $\I$-open. This means that $(Z,\I,E)$ is not soft connected, which is impossible. Hence $(Y,\I_Y,E)$ must be maximal soft connected.
\end{proof}

\begin{defn}\cite{ilango2013soft}
A soft topological space $(Z,\I,E)$ is called soft submaximal if each soft $\I$-dense set is soft $\I$-open. 
\end{defn}

\begin{lem}\label{remark2.2}
Let $(Z,\I,E)$ be a soft connected space and let $\I^*=\I[Y]$ be an s-extension of $\I$ over $Z$. If $(Y,E)$ is soft $\I$-dense, then $(Z,\I^*,E)$ is soft connected
\end{lem}
\begin{proof}
It follows from Theorem 3.17 \cite{ameen2021alghour}.
\end{proof}

\begin{lem}\label{maxcontosubmax}
If $(Z,\I,E)$ is a maximal soft connected space, then $(Z,\I,E)$ is soft submaximal.
\end{lem}
\begin{proof}
Let $(D,E)$ be a soft $\I$-dense set over $Z$. By Lemma \ref{remark2.2}, $(Z,\I^*,E)$ is soft connected, where $\I^*=\I[D]$, but $\I$ is soft maximal, hence $\I=\I^*$. Thus $(D,E)$ must be a soft $\I$-open.
\end{proof}

\begin{thm}\label{subspaceconnected}
If $(Y,E)$ is a soft connected subset of a maximal soft connected space $(Z,\I,E)$, then $(Y,\I_Y,E)$ is maximal soft connected.
\end{thm}
\begin{proof}
Since $(Y,E)$ is soft connected, then $\Cl_\I((Y,E))$ is also soft connected and, by Proposition \ref{l2},  $\Cl_\I((Y,E))$ is maximal. By Lemma \ref{maxcontosubmax}, $(Y,E)$ is soft open in $\Cl_\I((Y,E))$ because $(Y,E)$ is soft dense in $\Cl_\I((Y,E))$. By Proposition \ref{l1}, $(Y,E)$ is maximal soft connected.
\end{proof}

\section*{Conclusion}
The growth of topology has been supported by the continual supply of classes of topological spaces, examples, and their properties and relations.
As a result, extending the area of soft topological spaces in the same way is significant. We have shown that the collection of all soft topologies on a non-empty set is a complete lattice. The minimal element in this lattice is the soft indiscrete topology, which is both soft compact and soft connected. One might ask the question of what be will be the structure of maximal soft compact and maximal soft connected topologies. The third section answers this question and gives some more properties of these topologies. We have characterized maximal soft compact topologies in terms of soft closed-compact subsets and soft homeomorphisms. It is shown that the class of maximal soft compact topologies contains the class of stable soft compact $T_2$ topologies and is contained in the class of soft $T_1$ topologies. Furthermore, we have seen that the class of maximal soft connected topologies is contained in the class of soft $T_0$ topologies. With the help of a simple extension of a soft topology, we have found that any soft connected topology relativised with a maximal soft connected is also maximal.

As part of future work, the following tasks are expected to be completed: 

$\bullet$ Different soft point theories can be applied to all the results presented in this paper.

$\bullet$ Recall that a soft topology $\I$ on $Z$ with the property $\Gamma$ is $\Gamma$-maximal if any soft topology finer than $\I$ does not have the property $\Gamma$. One can examine different soft topological properties in place of soft compact or soft connected, namely: soft countably compact, soft sequentially compact, soft paracompact, soft Menger, soft path-connected, soft J-topologies, and so on. 

$\bullet$ This work is done on soft topologies, one can work on different topological structures, like fuzzy soft topologies, supra topologies, infra topologies, etc.


\end{document}